\documentclass{article}

\usepackage[T1]{fontenc}
\usepackage[utf8]{inputenc}
\usepackage{lmodern}
\usepackage{geometry}
\usepackage{microtype}
\usepackage{amsmath,amssymb,amsthm,mathtools}
\usepackage{booktabs}
\usepackage{enumitem}
\usepackage{tikz}
\usepackage{pgfplots}
\usetikzlibrary{calc,positioning,arrows.meta,patterns}
\usepgfplotslibrary{fillbetween}
\pgfplotsset{compat=1.18}
\usepackage[hidelinks]{hyperref}
\IfFileExists{orcidlink.sty}{\usepackage{orcidlink}}{\newcommand{\orcidlink}[1]{\href{https://orcid.org/#1}{\texttt{#1}}}}
\usepackage{PRIMEarxiv}

\definecolor{RTwoBlue}{HTML}{1967A9}
\definecolor{RTwoOrange}{HTML}{D87914}
\definecolor{RTwoGreen}{HTML}{2E8B57}
\definecolor{RTwoPurple}{HTML}{7A4EAB}
\definecolor{RTwoGray}{HTML}{555555}

\pgfplotsset{
  colormap={whitered}{color(0cm)=(white); color(1cm)=(gray!75!blue)}
}

\newcommand{\fillconeparallelogram}[9]{%
  \pgfmathsetmacro{\dx}{#7-#1}
  \pgfmathsetmacro{\dy}{#8-#2}
  \pgfmathsetmacro{\detuv}{#3*#6-#4*#5}
  \pgfmathsetmacro{\lamu}{(\dx*#6-\dy*#5)/(\detuv)}
  \pgfmathsetmacro{\lamv}{(#3*\dy-#4*\dx)/(\detuv)}
  \coordinate (P) at (#1,#2);
  \coordinate (B) at ($(P)+({\lamu*(#3)},{\lamu*(#4)})$);
  \coordinate (D) at ($(P)+({\lamv*(#5)},{\lamv*(#6)})$);
  \coordinate (R) at (#7,#8);
  \fill[#9] (P) -- (B) -- (R) -- (D) -- cycle;
}

\newcommand{\drawconehvpanel}[8]{%
  \begin{scope}[shift={(#1,0)}]
    \draw[thick, color=black!22, dotted, step=1cm] (-0.6,-0.6) grid (5,4);
    \draw[thick, ->] (-0.6,0) -- (5.45,0) node[right] {$f_1$};
    \draw[thick, ->] (0,-0.6) -- (0,4.25) node[above] {$f_2$};
    \foreach \i in {1,2,3,4} {
      \node[anchor=north] at (\i,0) {$\i$};
      \node[anchor=east] at (0,\i) {$\i$};
    }
    \node[anchor=north east] at (0,0) {$(0,0)$};
    \node[font=\large] at (2.2,4.55) {#2};

    \coordinate (y1) at (1,3.5);
    \coordinate (y2) at (2,2);
    \coordinate (y3) at (2.5,1.5);
    \coordinate (y4) at (3.5,1);
    \coordinate (y5) at (4,0.5);
    \coordinate (r)  at (5,4);

    \pgfmathsetmacro{\ux}{cos(#3)}
    \pgfmathsetmacro{\uy}{sin(#3)}
    \pgfmathsetmacro{\vx}{cos(#4)}
    \pgfmathsetmacro{\vy}{sin(#4)}

    \fillconeparallelogram{1}{3.5}{\ux}{\uy}{\vx}{\vy}{5}{4}{gray,opacity=0.28}
    \fillconeparallelogram{2}{2}{\ux}{\uy}{\vx}{\vy}{5}{4}{gray,opacity=0.28}
    \fillconeparallelogram{2.5}{1.5}{\ux}{\uy}{\vx}{\vy}{5}{4}{gray,opacity=0.28}
    \fillconeparallelogram{3.5}{1}{\ux}{\uy}{\vx}{\vy}{5}{4}{gray,opacity=0.28}
    \fillconeparallelogram{4}{0.5}{\ux}{\uy}{\vx}{\vy}{5}{4}{gray,opacity=0.28}

    \node at (3.0,3.0) {#8};

    \fill[red]    (y1) circle (4pt) node[left] {$y^{(1)}$};
    \fill[orange] (y2) circle (4pt) node[left] {$y^{(2)}$};
    \fill[teal]   (y3) circle (4pt) node[left] {$y^{(3)}$};
    \fill[blue]   (y4) circle (4pt) node[left] {$y^{(4)}$};
    \fill[violet] (y5) circle (4pt) node[left] {$y^{(5)}$};
    \fill (r) circle (4pt) node[above] {$r$};

    \coordinate (p) at (#6,#7);
    \draw[dashed,thick] (p) -- ($(p)+(1.75*\ux,1.75*\uy)$);
    \draw[dashed,thick] (p) -- ($(p)+(1.75*\vx,1.75*\vy)$);
    \pgfmathsetmacro{\midang}{0.5*(#3+#4)}
    \draw[thick] ($(p)+(#3:0.55)$) arc[start angle=#3,end angle=#4,radius=0.55];
    \node at ($(p)+(\midang:0.85)$) {$\gamma$};

    \node at (1.1,2.7) {$Y$};
    \node at (0.9,-0.35) {#5};
  \end{scope}%
}

\numberwithin{equation}{section}

\newtheorem{theorem}{Theorem}[section]
\newtheorem{proposition}[theorem]{Proposition}

\newtheorem{corollary}[theorem]{Corollary}
\newtheorem{definition}[theorem]{Definition}

\newcommand{\R}{\mathbb{R}}
\newcommand{\E}{\mathbb{E}}
\newcommand{\Prob}{\mathbb{P}}

\newcommand{\HV}{\operatorname{HV}}
\newcommand{\HVI}{\operatorname{HVI}}
\newcommand{\EHVI}{\operatorname{EHVI}}
\newcommand{\TEHVI}{\operatorname{TEHVI}}
\newcommand{\WHV}{\operatorname{WHV}}
\newcommand{\CHV}{\operatorname{CHV}}

\newcommand{\ERII}{\operatorname{ER2I}}
\newcommand{\TSM}{\operatorname{TSM}}
\newcommand{\EI}{\operatorname{EI}}
\newcommand{\Dom}{\operatorname{Dom}}
\newcommand{\dd}{\,\mathrm{d}}

\newcommand{\DeltaM}{\Delta^{m-1}}
\newcommand{\1}{\mathbf{1}}

\title{Preference-Shaped Expected Hypervolume and R2 Improvement: Exact Computation and Monotonicity}
\author{Michael T. M. Emmerich~\orcidlink{0000-0002-7342-2090}\\
Faculty of Information Technology, University of Jyväskylä, Finland}
\date{}

\begin{document}
\maketitle

\begin{abstract}
This paper studies preference-shaped expected improvement criteria for Bayesian multiobjective optimization. We consider two indicator families which are often used for similar algorithmic purposes, but which are geometrically different. The hypervolume indicator is based on a dystopian reference point and measures dominated volume in objective space. The \(R_2\) indicator is based on a utopian point and evaluates approximation sets through weighted Tchebycheff scalarization envelopes. The purpose of the paper is to make precise which preference transformations preserve exact computation, Pareto compatibility, and monotonicity properties, and which transformations change the underlying geometry. On the hypervolume side, we revisit canonical EHVI through the particular-hypervolume-improvement (PHVI) representation of Deng et al. 2025, formulate product-density weighted EHVI in desirability coordinates, discuss cone-based EHVI as ordinary EHVI after a linear cone transformation, and separate these cases from truncated EHVI, where variance monotonicity may fail. On the \(R_2\) side, we prove that exact integral \(R_2\) improvement is not, in general, an ordinary objective-space weighted hypervolume. The obstruction is lower-dimensional: Lebesgue-density hypervolume cannot see certain boundary contributions that Tchebycheff scalarizations still detect. We then show that exact integral \(R_2\) improvement is exactly a scalarization-space volume, namely the measure of the Tchebycheff shadow between the incumbent scalarization envelope and the reference envelope. This representation yields finite-sum ER2I algorithms for discrete \(R_2\), quadrature methods for exact integral \(R_2\), and an achievement-space Gaussian surrogate formulation in which ER2I is an integral of scalar Gaussian expected improvements.
\end{abstract}


\section{Introduction}
\label{sec:introduction}

Bayesian multiobjective optimization is concerned with the sequential optimization of expensive vector-valued black-box functions. Since each evaluation may be costly, one uses a probabilistic surrogate model to decide where to evaluate next. A common strategy is to transform the predictive distribution at a candidate point into the expected improvement of a set quality indicator. The quality indicator determines what kind of progress is rewarded; the probabilistic expectation determines how uncertainty is exploited. Thus, before one studies numerical algorithms, one must understand the geometry of the underlying indicator.

The first indicator geometry considered in this paper is hypervolume. For a finite approximation set, the hypervolume indicator measures the size of the objective-space region dominated by the set and bounded by a reference point. It is Pareto-compliant and has a clear geometric interpretation. Its expected-improvement version, EHVI, was introduced in Emmerich's doctoral dissertation and subsequently developed in Gaussian-process-assisted evolutionary optimization \cite{Emmerich2005,EmmerichGiannakoglouNaujoks2006}. In EHVI, the deterministic hypervolume improvement of a candidate objective vector is averaged with respect to the predictive distribution. This gives an acquisition function which combines set quality and uncertainty in a principled way.

The second indicator geometry is the \(R_2\) geometry. The original \(R_2\) indicator of Hansen and Jaszkiewicz evaluates approximation sets by applying scalarizing functions for a finite collection of weights \cite{HansenJaszkiewicz1994}. Later work studied its indicator properties in detail \cite{BrockhoffWagnerTrautmann2012}. In contrast to hypervolume, \(R_2\) is naturally anchored at a utopian or ideal point and works through scalarization values rather than through dominated objective-space volume. Expected \(R_2\)-indicator improvement was introduced for Bayesian multiobjective optimization by Deutz, Emmerich, and Yang \cite{DeutzEmmerichYang2019}. More recent exact-integral variants replace a finite weight set by integration over the weight simplex and thereby recover stronger Pareto-compliance properties \cite{SchaepermeierKerschke2024,JaszkiewiczZielniewicz2024}.

The present paper compares these two geometries under preference shaping. Preference information may enter through density weights, desirability maps, cone orders, truncation to a region of interest, scalarizing weights, or scalarization functions. However, a preference-shaped acquisition should not be judged only by its intuitive modelling appeal. It should also be examined structurally. Following the expected-improvement analysis of Wagner, Emmerich, Deutz, and Ponweiser \cite{WagnerEmmerichDeutzPonweiser2010}, we use three guiding questions:
\begin{enumerate}[label=(G\arabic*),leftmargin=1.0cm]
\item \emph{Computation.} Does the shaped expected improvement admit an exact formula, a reduction to known hypervolume computations, or a controlled numerical integration scheme?
\item \emph{Compatibility.} Does the construction preserve Pareto compatibility, or the appropriate preference-order analogue?
\item \emph{Monotonicity.} Is the acquisition monotone with respect to favourable changes of the predictive mean and, where appropriate, with respect to increasing predictive variance?
\end{enumerate}
These questions are deliberately elementary. They are also restrictive enough to reveal important differences between superficially similar acquisition functions.

On the hypervolume side, the central computational observation is the PHVI representation theorem, due to Deng, Sun, Zhang, and Li: for independent Gaussian predictive coordinates, EHVI is a particular ordinary hypervolume improvement after coordinate-wise expected-improvement transformations \cite{DengSunZhangLi2025}. This theorem is more than a computational trick. It explains why exact hypervolume algorithms and monotonicity arguments can be transferred to EHVI. It also provides a useful test for preference transformations: if the transformation leads back to ordinary hypervolume geometry, much of the existing theory remains available.

On the \(R_2\) side, the situation is more subtle. Although \(R_2\)-based methods have been used to approximate hypervolume contributions \cite{ShangIshibuchiNi2018}, exact integral \(R_2\) improvement is not ordinary objective-space weighted hypervolume. The reason is not a matter of notation but of dimension. An objective-space hypervolume density with respect to Lebesgue measure cannot assign positive value to lower-dimensional boundary regions. Tchebycheff scalarizations, however, may still register improvement there. The magnitude indicator of dominated sets partly addresses such lower-dimensional visibility by adding coordinate-projection terms \cite{Emmerich2026Magnitude}, but standard metric magnitude still does not provide a bijection with exact \(R_2\). The correct representation is instead a scalarization-envelope representation: for every weight, one measures the interval between the incumbent scalarization envelope and the reference envelope. The union of these intervals over the weight simplex is what we call the Tchebycheff shadow.

The scope of the paper is therefore as follows. On the EHVI side, we discuss canonical EHVI, product-density and desirability transformations, cone-based transformations, and truncation to a region of interest. On the \(R_2\) side, we discuss exact scalarization-space volume, the obstruction to objective-space weighted hypervolume, metric magnitude, Tchebycheff shadows, and ER2I computation with achievement-space Gaussian surrogates. The main contributions are:
\begin{enumerate}[label=(C\arabic*),leftmargin=1.0cm]
\item We give a unified account of canonical EHVI through the PHVI representation and record the resulting dimension-free monotonicity interpretation.
\item We formulate product-density weighted EHVI as ordinary hypervolume in desirability coordinates and state the corresponding positive-kernel condition for Pareto compliance.
\item We distinguish cone transformations, which preserve ordinary EHVI structure after a linear change of coordinates, from truncation, which focuses preferences but may destroy variance monotonicity.
\item We prove that exact integral \(R_2\) improvement is not, in general, ordinary objective-space weighted hypervolume.
\item We show that exact integral \(R_2\) improvement is exactly the scalarization-space volume of a Tchebycheff shadow.
\item We derive ER2I computation rules: finite scalar expected-improvement sums for discrete \(R_2\), quadrature over the weight simplex for exact integral \(R_2\), and an achievement-space Gaussian formulation in which the integrand is scalar Gaussian EI.
\end{enumerate}

The rest of the paper is organized as follows. Section~\ref{sec:prelim} fixes the minimization convention, introduces hypervolume improvement, EHVI, discrete \(R_2\), exact integral \(R_2\), and their corresponding expected-improvement forms. Section~\ref{sec:phvi} recalls canonical EHVI and the PHVI representation that turns EHVI into ordinary HVI after coordinate-wise EI transformations. Section~\ref{sec:preference-ehvi} studies three preference mechanisms for EHVI: product-density/desirability transformations, cone transformations, and truncation to a region of interest. Section~\ref{sec:r2-volume} proves the scalarization-space representation of exact integral \(R_2\) improvement and the obstruction to ordinary objective-space weighted hypervolume. Section~\ref{sec:magnitude-shadows} compares this obstruction with magnitude-type lower-dimensional terms and introduces Tchebycheff shadows. Section~\ref{sec:er2i-surrogates} derives ER2I formulas for achievement-space Gaussian surrogates, while Section~\ref{sec:algorithms} translates the formulas into discrete-weight and quadrature-based algorithms. Section~\ref{sec:conclusions} summarizes the results using the three guiding modes. The appendices collect the proofs and a list of symbols.

\section{Notation and preliminaries}
\label{sec:prelim}

We first fix notation. A minimization convention is used throughout, because it makes the relation between Pareto dominance, reference points, and Tchebycheff scalarizations explicit. The notation is deliberately uniform: hypervolume quantities are written as volumes of dominated regions, while \(R_2\) quantities are written as integrals of scalarization envelopes. This common notation makes the later comparison possible without changing conventions between sections.

We write objective vectors in minimization orientation unless stated otherwise. For a finite approximation set $A\subset\R^m$ and reference point $r$, the dominated region is
\[
\Dom_r(A)=\bigcup_{a\in A} [a_1,r_1]\times\cdots\times [a_m,r_m].
\]
The hypervolume indicator is $\HV_r(A)=\int_{\Dom_r(A)}1\dd y$. A weighted hypervolume with density $K$ is
\[
\WHV_{K,r}(A)=\int_{\Dom_r(A)}K(y)\dd y.
\]
For a new point $y$, hypervolume improvement is
\[
I_A(y)=\HVI_r(\{y\},A)=\HV_r(A\cup\{y\})-\HV_r(A),
\]
and for a random predictive vector $Y$, $\EHVI_A=\E[I_A(Y)]$. Equivalently, if the predictive distribution at a design $x$ has density $p_x$, then canonical EHVI is the integral
\[
\EHVI_A(x)=\int_{\Omega_r} I_A(y)\,p_x(y)\dd y.
\]
Thus EHVI is the expectation of deterministic hypervolume improvement, not a different deterministic indicator. In the Gaussian independent-coordinate case considered below, this integral can be evaluated exactly by the PHVI representation: the expectation is absorbed into coordinate-wise EI transforms and the remaining computation is ordinary hypervolume improvement. Figures~\ref{fig:ehvi-hvi-2d} and~\ref{fig:ehvi-predictive-distribution} illustrate these two levels: deterministic HVI for a realized objective vector and EHVI as its average over a predictive density.

For the $R_2$ side, let $z^+$ be a utopian point and $\lambda\in\DeltaM=\{\lambda\in\R_+^m:\sum_i\lambda_i=1\}$. The weighted Tchebycheff scalarization is
\[
g_\lambda(y;z^+)=\max_i \lambda_i(y_i-z_i^+).
\]
For a finite set $A$, define the scalarization envelope
\[
h_A(\lambda)=\min_{a\in A}g_\lambda(a;z^+).
\]
We use two $R_2$ variants. The discrete $R_2$ indicator is based on a finite weight set $\Lambda_K=\{\lambda^{(1)},\ldots,\lambda^{(K)}\}$ and evaluates
\[
R_{2,K}(A)=\frac{1}{K}\sum_{k=1}^K h_A(\lambda^{(k)}),
\qquad
I_{R_{2,K}}(A;r)=\frac{1}{K}\sum_{k=1}^K\bigl(h_r(\lambda^{(k)})-h_A(\lambda^{(k)})\bigr)_+.
\]
Its expected-improvement version is the discrete ER2I acquisition
\[
\ERII_{K,A}(x)=\frac{1}{K}\sum_{k=1}^K
\E\bigl[(h_A(\lambda^{(k)})-g_{\lambda^{(k)}}(Y(x);z^+))_+\bigr].
\]
The exact integral $R_2$ value replaces the finite weight set by integration over the simplex,
\[
R_2(A)=\int_{\DeltaM} h_A(\lambda)\rho(\lambda)\dd\sigma(\lambda),
\]
where $\rho$ is a nonnegative weight density. With a reference point $r$, the integral improvement form is
\[
I_{R_2}(A;r)=\int_{\DeltaM}\bigl(h_r(\lambda)-h_A(\lambda)\bigr)_+\rho(\lambda)\dd\sigma(\lambda),
\]
and its expected-improvement version, called exact or integral ER2I, is obtained by taking the predictive expectation of the integrand. Figures~\ref{fig:r2-objective-view} and~\ref{fig:r2-weight-view} contrast the objective-space Tchebycheff construction with the corresponding weight-space envelope; Figure~\ref{fig:tsm-2d-map} previews the scalarization-space shadow used later to represent exact integral $R_2$ improvement.

\begin{figure}[t]
\centering
\begin{tikzpicture}[scale=0.95]
  \draw[thick, color=black!22, dotted, step=1cm] (-0.6,-0.6) grid (5,4);
  \draw[thick, ->] (-0.6,0) -- (5.45,0) node[right] {$f_1$};
  \draw[thick, ->] (0,-0.6) -- (0,4.25) node[above] {$f_2$};
  \foreach \i in {1,2,3,4} {
    \node[anchor=north] at (\i,0) {$\i$};
    \node[anchor=east] at (0,\i) {$\i$};
  }
  \node[anchor=north east] at (0,0) {$(0,0)$};

  \fill[gray,opacity=0.28] (1,3.5) -- (1,4) -- (5,4) -- (5,3.5) -- cycle;
  \fill[gray,opacity=0.28] (2,2.5) -- (2,3.5) -- (5,3.5) -- (5,2.5) -- cycle;
  \fill[gray,opacity=0.28] (3,1.5) -- (3,2.5) -- (5,2.5) -- (5,1.5) -- cycle;

  \fill[red,opacity=0.20] (2.45,1.05) -- (2.45,2.5) -- (3,2.5) -- (3,1.5) -- (5,1.5) -- (5,1.05) -- cycle;
  \draw[red!70!black,thick,dashed] (2.45,1.05) -- (2.45,2.5) -- (3,2.5) -- (3,1.5) -- (5,1.5) -- (5,1.05) -- cycle;

  \draw[black,ultra thick] (1,4) -- (1,3.5) -- (2,3.5) -- (2,2.5) -- (3,2.5) -- (3,1.5) -- (4,1.5);
  \draw[black,dashed] (2.45,1.05) -- (5,1.05) -- (5,4);

  \fill[red] (1,3.5) circle (4pt) node[left] {$y^{(1)}$};
  \fill[orange] (2,2.5) circle (4pt) node[left] {$y^{(2)}$};
  \fill[teal] (3,1.5) circle (4pt) node[left] {$y^{(3)}$};
  \fill[blue] (2.45,1.05) circle (4pt) node[below left] {$y$};
  \fill (5,4) circle (4pt) node[above] {$r$};

  \node at (3.6,3.25) {$\HV_r(A)$};
  \node[red!70!black] at (3.85,1.25) {$I_A(y)$};
\end{tikzpicture}
\caption{Hypervolume improvement in two objectives. The grey rectangles indicate the hypervolume already dominated by the incumbent approximation set $A=\{y^{(1)},y^{(2)},y^{(3)}\}$ and bounded by the reference point $r$. The red shaded region is the additional hypervolume improvement $I_A(y)$ contributed by a newly realized candidate vector $y$. The orientation shown is the classical minimization orientation; the maximization convention in the text is obtained by reversing signs.}
\label{fig:ehvi-hvi-2d}
\end{figure}
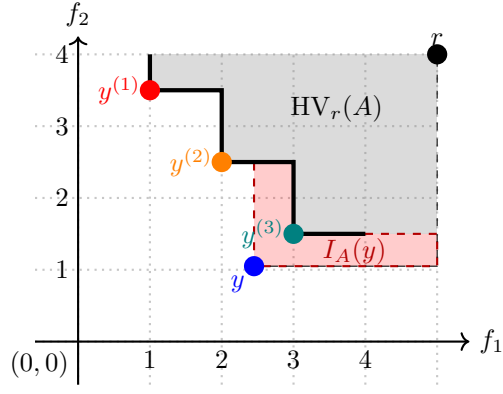

\begin{figure}[t]
\centering
\begin{tikzpicture}[
  scale=1,
  declare function={muone=2;},
  declare function={mutwo=1.5;},
  declare function={sigmaone=0.7;},
  declare function={sigmatwo=0.6;},
  declare function={normal(\m,\s)=1/(\s*sqrt(2*pi))*exp(-((x-\m)^2)/(2*\s^2));},
  declare function={bivar(\ma,\sa,\mb,\sb)=1/(2*pi*\sa*\sb)*exp(-0.5*(((x-\ma)^2)/(\sa^2)+((y-\mb)^2)/(\sb^2)));}
]
\begin{axis}[
  colormap name=whitered,
  width=0.75\textwidth,
  height=0.52\textwidth,
  view={45}{65},
  enlargelimits=false,
  grid=major,
  domain=0:4,
  y domain=0:4,
  samples=24,
  xlabel={$y_1$},
  ylabel={$y_2$},
  zlabel={$p_{\mu,\sigma}(y_1,y_2)$},
  colorbar,
  colorbar style={
    at={(1,1)},
    anchor=north west,
    height=0.20*\pgfkeysvalueof{/pgfplots/parent axis height},
    title={$p_{\mu,\sigma}$}
  }
]

\draw[black,ultra thick] (axis cs:1,4,0) -- (axis cs:1,3.5,0) -- (axis cs:2,3.5,0) -- (axis cs:2,2.5,0) -- (axis cs:3,2.5,0) -- (axis cs:3,1.5,0) -- (axis cs:4,1.5,0);
\draw[black,dashed,fill=gray!40,opacity=0.6] (axis cs:4,4,0) -- (axis cs:1,4,0) -- (axis cs:1,3.5,0) -- (axis cs:2,3.5,0) -- (axis cs:2,2.5,0) -- (axis cs:3,2.5,0) -- (axis cs:3,1.5,0) -- (axis cs:4,1.5,0) -- cycle;

\addplot3[surf,opacity=0.50] {bivar(muone,sigmaone,mutwo,sigmatwo)};
\addplot3[domain=-0.5:4,samples=31,samples y=0,thick,smooth] (x,4,{normal(muone,sigmaone)});
\addplot3[domain=-0.5:4,samples=31,samples y=0,thick,smooth] (-1,x,{normal(mutwo,sigmatwo)});

\draw[black!50] (axis cs:-1,0,0) -- (axis cs:4,0,0);
\draw[black!50,dashed] (axis cs:1,0,0) -- (axis cs:1,4,0);
\draw[black!50,dashed] (axis cs:2,0,0) -- (axis cs:2,4,0);
\draw[black!50,dashed] (axis cs:3,0,0) -- (axis cs:3,4,0);
\draw[black!50,dashed] (axis cs:0,3.5,0) -- (axis cs:4,3.5,0);
\draw[black!50,dashed] (axis cs:0,2.5,0) -- (axis cs:4,2.5,0);
\draw[black!50,dashed] (axis cs:0,1.5,0) -- (axis cs:4,1.5,0);
\draw[black!50] (axis cs:0,-1,0) -- (axis cs:0,4,0);

\addplot3[only marks,mark=*,mark size=2pt,color=red] coordinates {(muone,mutwo,0)};
\node at (axis cs:2,1.5,0) [pin=-120:{$\hat y(x)$}] {};
\draw[thick,gray,dashed] (axis cs:2,1.5,0) ellipse [x radius=sigmaone,y radius=sigmatwo,rotate=0];
\node at (axis cs:2.7,1.5,0) [pin=30:{$\hat s(x)$}] {};

\node at (axis cs:4,4,0) [pin=-180:{$r$}] {};
\node at (axis cs:1,3.5,0) [pin=0:{$y^{(1)}$}] {};
\node at (axis cs:2,2.5,0) [pin=0:{$y^{(2)}$}] {};
\node at (axis cs:3,1.5,0) [pin=0:{$y^{(3)}$}] {};

\end{axis}
\end{tikzpicture}
\caption{Illustration of a predictive distribution behind EHVI in two objectives. The incumbent front approximation is $\{y^{(1)},y^{(2)},y^{(3)}\}$; the grey region is the reference-bounded hypervolume. The red point is the predictive mean $\hat y(x)$ and the dashed ellipse indicates the coordinate-wise standard deviations $\hat s(x)$. EHVI integrates the hypervolume improvement over the predictive density surface.}
\label{fig:ehvi-predictive-distribution}
\end{figure}
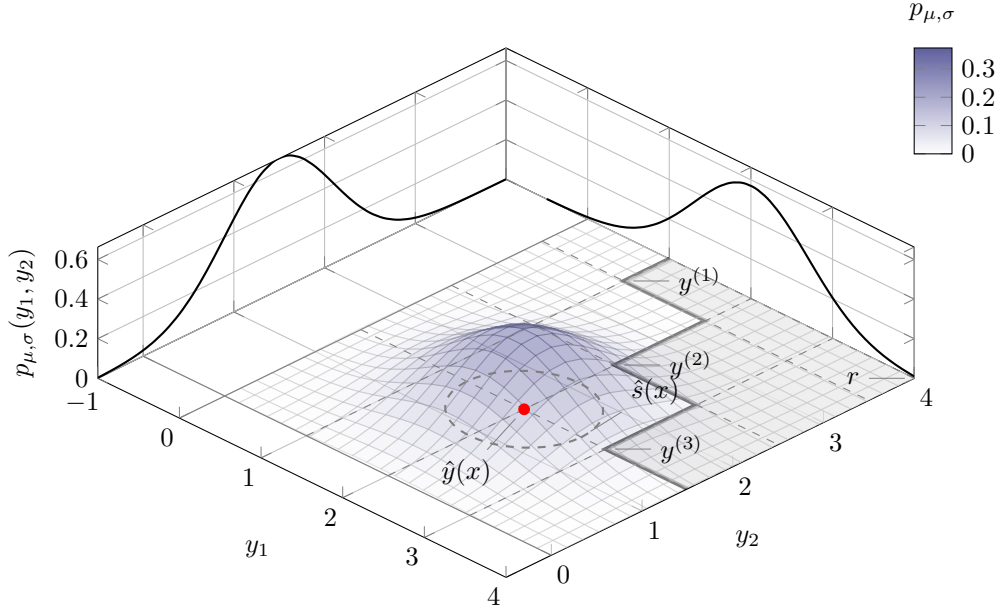

\begin{figure}[t]
\centering
\begin{minipage}{\textwidth}
\begin{tikzpicture}
\begin{axis}[
  width=\linewidth,
  height=0.36\textwidth,
  xmin=-5, xmax=55,
  ymin=-0.25, ymax=1.25,
  axis lines=left,
  xlabel={$y_i$},
  ylabel={$D_i(y_i)$},
  grid=both,
  grid style={line width=.1pt, draw=gray!18},
  major grid style={line width=.2pt, draw=gray!32},
  xtick={5,30},
  xticklabels={$a_i$,$r_i$},
  ytick={-0.2,0,1,1.2},
  tick label style={font=\scriptsize},
  label style={font=\small},
  clip=false
]
  \addplot[draw=none, fill=RTwoGreen!10] coordinates {(-5,-0.25) (5,-0.25) (5,1.25) (-5,1.25)} \closedcycle;
  \addplot[draw=none, fill=RTwoOrange!10] coordinates {(5,-0.25) (30,-0.25) (30,1.25) (5,1.25)} \closedcycle;
  \addplot[draw=none, fill=red!7] coordinates {(30,-0.25) (55,-0.25) (55,1.25) (30,1.25)} \closedcycle;

  \addplot[very thick, RTwoBlue, domain=-5:5, samples=80] {1/(x-10)+1.2};
  \addplot[very thick, RTwoBlue, domain=5:30, samples=80] {1.2-0.04*x};
  \addplot[very thick, RTwoBlue, domain=30:55, samples=80] {1/(x-25)-0.2};

  \draw[dashed, thick] (axis cs:5,-0.25) -- (axis cs:5,1.25);
  \draw[dashed, thick] (axis cs:30,-0.25) -- (axis cs:30,1.25);
  \addplot[only marks, mark=*, mark size=2pt, RTwoBlue] coordinates {(5,1) (30,0)};
  \node[font=\scriptsize, anchor=south west] at (axis cs:-4,1.15) {overachievement};
  \node[font=\scriptsize, anchor=south] at (axis cs:17.5,0.45) {desirable range};
  \node[font=\scriptsize, anchor=south east] at (axis cs:54,-0.12) {underachievement};
  \node[font=\scriptsize, anchor=west] at (axis cs:6,1.02) {$D_i(a_i)=1$};
  \node[font=\scriptsize, anchor=west] at (axis cs:31,0.04) {$D_i(r_i)=0$};
\end{axis}
\end{tikzpicture}
\end{minipage}

\caption{D-PHI desirability function with aspiration level \(a_i\), reservation level \(r_i\), and marginal product-kernel interpretation \cite{LiangShavazipourSainiEmmerich2026}. The D-PHI construction uses aspiration and reservation levels to define coordinate-wise desirability maps
\(
T_i(y_i)=D_i(y_i),\qquad D_i(a_i)=1,\qquad D_i(r_i)=0.
\)
The local slopes \(D_i'\) define marginal product-kernel factors and hence shape weighted hypervolume improvement.}
\label{fig:dphi-desirability}
\end{figure}
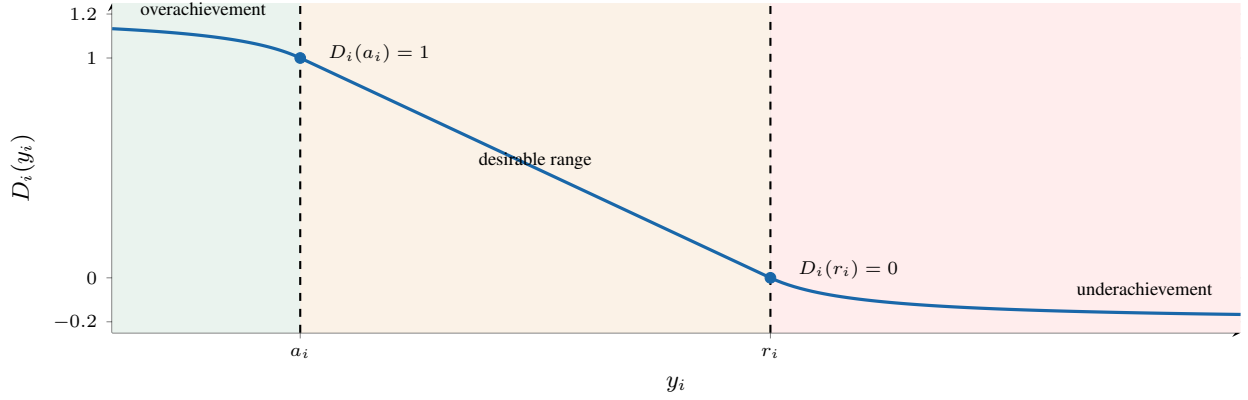

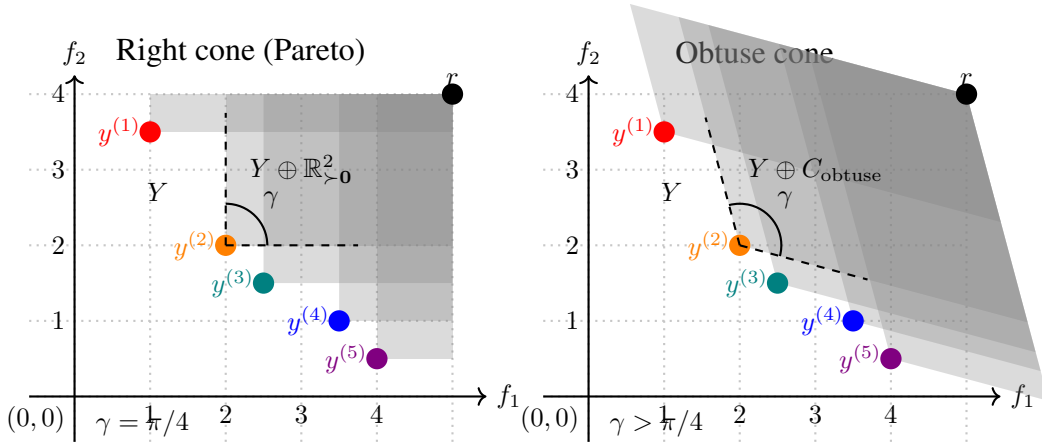
\begin{figure}[t]
\centering
\begin{tikzpicture}[scale=1]
  \drawconehvpanel{0}{Right cone (Pareto)}{0}{90}{$\gamma=\pi/4$}{2}{2}{$Y\oplus \mathbb{R}^2_{\succ \mathbf{0}}$}
  \drawconehvpanel{6.8}{Obtuse cone}{-15}{105}{$\gamma>\pi/4$}{2}{2}{$Y\oplus C_{\mathrm{obtuse}}$}
\end{tikzpicture}
\caption{Two-dimensional illustration of cone-based hypervolume regions. The left panel shows the classical right cone corresponding to Pareto dominance, while the right panel shows an obtuse cone. In both panels, the shaded region is the cone-dominated region generated by the sample set $Y=\{y^{(1)},\ldots,y^{(5)}\}$ and bounded by the reference point $r$. The opening angle $\gamma$ determines the underlying preference geometry.}
\label{fig:cone-hv-2d}
\end{figure}

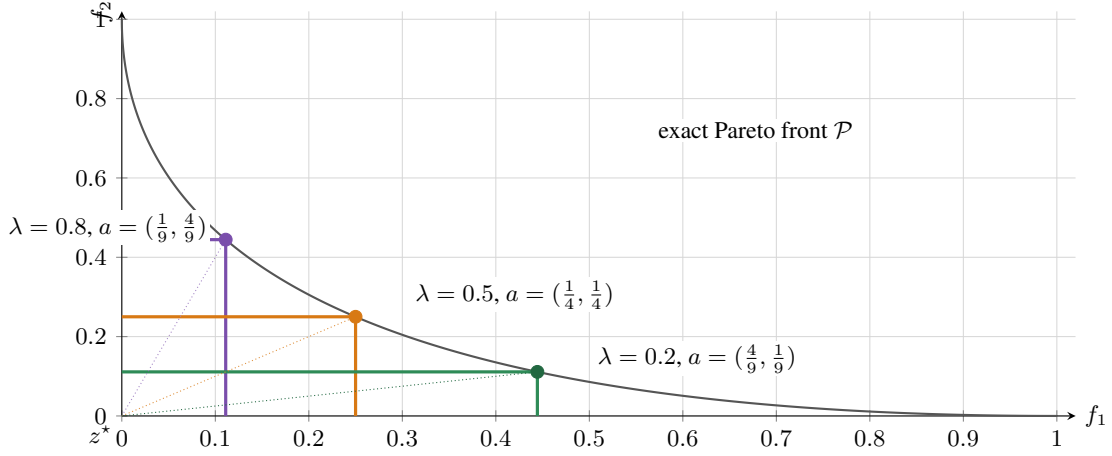
\begin{figure}[t]
\centering
\begin{tikzpicture}
\begin{axis}[
  width=0.86\textwidth,
  height=0.42\textwidth,
  axis lines=left,
  xmin=0, xmax=1.02,
  ymin=0, ymax=1.02,
  xlabel={$f_1$}, ylabel={$f_2$},
  xlabel style={at={(axis description cs:1,0)},anchor=west},
  ylabel style={at={(axis description cs:0,1)},anchor=south},
  grid=both,
  grid style={line width=.1pt, draw=gray!18},
  major grid style={line width=.2pt, draw=gray!32},
  tick label style={font=\small},
  label style={font=\small},
  clip=false
]
  \addplot[RTwoGray, thick, domain=0:1, samples=200] ({x^2},{(1-x)^2});

  \draw[RTwoPurple, very thick] (axis cs:0,0.4444444444) -- (axis cs:0.1111111111,0.4444444444) -- (axis cs:0.1111111111,0);
  \draw[RTwoOrange, very thick] (axis cs:0,0.2500000000) -- (axis cs:0.2500000000,0.2500000000) -- (axis cs:0.2500000000,0);
  \draw[RTwoGreen, very thick] (axis cs:0,0.1111111111) -- (axis cs:0.4444444444,0.1111111111) -- (axis cs:0.4444444444,0);

  \addplot[RTwoPurple!70, densely dotted] coordinates {(0,0) (0.1111111111,0.4444444444)};
  \addplot[RTwoOrange!85, densely dotted] coordinates {(0,0) (0.2500000000,0.2500000000)};
  \addplot[RTwoGreen!75!black, densely dotted] coordinates {(0,0) (0.4444444444,0.1111111111)};

  \addplot[only marks, mark=*, mark size=2.4pt, RTwoPurple] coordinates {(0.1111111111,0.4444444444)};
  \addplot[only marks, mark=*, mark size=2.4pt, RTwoOrange] coordinates {(0.2500000000,0.2500000000)};
  \addplot[only marks, mark=*, mark size=2.4pt, RTwoGreen!75!black] coordinates {(0.4444444444,0.1111111111)};

  \node[font=\small, fill=white, inner sep=1.2pt, anchor=east] at (axis cs:0.095,0.475) {$\lambda=0.8$, $a=(\frac19,\frac49)$};
  \node[font=\small, fill=white, inner sep=1.2pt, anchor=west] at (axis cs:0.31,0.305) {$\lambda=0.5$, $a=(\frac14,\frac14)$};
  \node[font=\small, fill=white, inner sep=1.2pt, anchor=west] at (axis cs:0.505,0.145) {$\lambda=0.2$, $a=(\frac49,\frac19)$};
  \node[font=\small, fill=white, inner sep=1.2pt, anchor=south west] at (axis cs:0.57,0.69) {exact Pareto front $\mathcal P$};
  \node[font=\small, anchor=north east] at (axis cs:0,0) {$z^\star$};
\end{axis}
\end{tikzpicture}
\caption{Objective-space view of the weighted Tchebycheff construction for the exact front $\mathcal P$. The three highlighted weights are representative examples. For each weight, the best front point is the point where the weighted rectangle touches the Pareto front.}
\label{fig:r2-objective-view}
\end{figure}

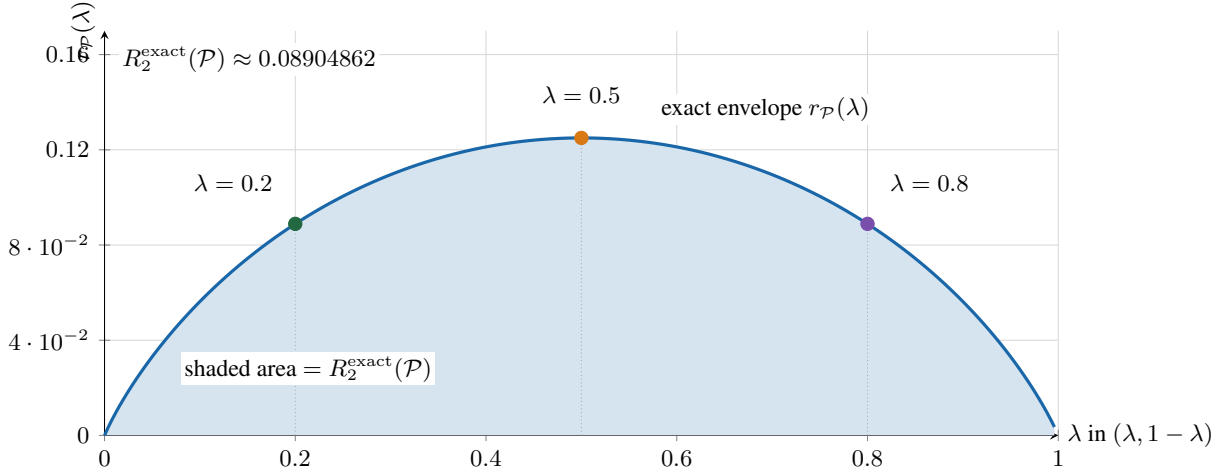
\begin{figure}[t]
\centering
\begin{tikzpicture}
\begin{axis}[
  width=0.86\textwidth,
  height=0.42\textwidth,
  axis lines=left,
  xmin=0, xmax=1,
  ymin=0, ymax=0.17,
  xlabel={$\lambda$ in $(\lambda,1-\lambda)$}, ylabel={$r_{\mathcal P}(\lambda)$},
  xlabel style={at={(axis description cs:1,0)},anchor=west},
  ylabel style={at={(axis description cs:0,1)},anchor=south},
  xtick={0,0.2,0.4,0.6,0.8,1.0},
  ytick={0,0.04,0.08,0.12,0.16},
  grid=both,
  grid style={line width=.1pt, draw=gray!18},
  major grid style={line width=.2pt, draw=gray!32},
  tick label style={font=\small},
  label style={font=\small},
  clip=false
]
  \addplot[draw=none, fill=RTwoBlue!18, domain=0:1, samples=250]
    {x*(1-x)/(sqrt(x)+sqrt(1-x))^2} \closedcycle;
  \addplot[RTwoBlue, very thick, domain=0:1, samples=250]
    {x*(1-x)/(sqrt(x)+sqrt(1-x))^2};

  \draw[gray!65, densely dotted] (axis cs:0.2,0) -- (axis cs:0.2,0.0888888889);
  \draw[gray!65, densely dotted] (axis cs:0.5,0) -- (axis cs:0.5,0.1250000000);
  \draw[gray!65, densely dotted] (axis cs:0.8,0) -- (axis cs:0.8,0.0888888889);

  \addplot[only marks, mark=*, mark size=2.5pt, RTwoGreen!75!black] coordinates {(0.2,0.0888888889)};
  \addplot[only marks, mark=*, mark size=2.5pt, RTwoOrange] coordinates {(0.5,0.1250000000)};
  \addplot[only marks, mark=*, mark size=2.5pt, RTwoPurple] coordinates {(0.8,0.0888888889)};

  \node[font=\small, fill=white, inner sep=1.2pt, anchor=south east] at (axis cs:0.18,0.101) {$\lambda=0.2$};
  \node[font=\small, fill=white, inner sep=1.2pt, anchor=south] at (axis cs:0.5,0.138) {$\lambda=0.5$};
  \node[font=\small, fill=white, inner sep=1.2pt, anchor=south west] at (axis cs:0.82,0.101) {$\lambda=0.8$};

  \node[font=\small, fill=white, inner sep=1.2pt, anchor=west] at (axis cs:0.58,0.137)
    {exact envelope $r_{\mathcal P}(\lambda)$};
  \node[font=\small, fill=white, inner sep=1.2pt, anchor=west] at (axis cs:0.08,0.028)
    {shaded area $=R_2^{\mathrm{exact}}(\mathcal P)$};
  \node[font=\small, fill=white, inner sep=1.2pt, anchor=north west] at (axis cs:0.015,0.165)
    {$R_2^{\mathrm{exact}}(\mathcal P)\approx 0.08904862$};
\end{axis}
\end{tikzpicture}
\caption{Weight-space view of the exact $R_2$ construction. The function $r_{\mathcal P}(\lambda)$ is the best weighted Tchebycheff value attainable on the exact front. The integrated exact $R_2$ value is the area under this curve.}
\label{fig:r2-weight-view}
\end{figure}

\begin{figure}[t]
\centering
\begin{tikzpicture}
\begin{axis}[
  name=obj,
  width=0.47\textwidth,
  height=0.42\textwidth,
  xmin=0,xmax=1.05,ymin=0,ymax=1.05,
  axis lines=left,
  xlabel={$f_1$},
  ylabel={$f_2$},
  grid=both,
  tick label style={font=\scriptsize},
  label style={font=\small},
  title={Objective space: four nondominated points},
  title style={font=\small},
  clip=false
]
\addplot[only marks, mark=*, mark size=2.2pt] coordinates {
  (0.12,0.88)
  (0.30,0.55)
  (0.55,0.32)
  (0.84,0.16)
};
\node[anchor=south west,font=\scriptsize] at (axis cs:0.12,0.88) {$a^{(1)}$};
\node[anchor=south west,font=\scriptsize] at (axis cs:0.30,0.55) {$a^{(2)}$};
\node[anchor=south west,font=\scriptsize] at (axis cs:0.55,0.32) {$a^{(3)}$};
\node[anchor=south west,font=\scriptsize] at (axis cs:0.84,0.16) {$a^{(4)}$};
\addplot[only marks, mark=square*, mark size=2.2pt] coordinates {(1,1)};
\node[anchor=south west,font=\scriptsize] at (axis cs:1,1) {$r$};
\addplot[only marks, mark=triangle*, mark size=2.2pt] coordinates {(0,0)};
\node[anchor=north east,font=\scriptsize] at (axis cs:0,0) {$z^+$};
\draw[dashed] (axis cs:0.12,0.88) -- (axis cs:0.30,0.55) -- (axis cs:0.55,0.32) -- (axis cs:0.84,0.16);
\end{axis}

\begin{axis}[
  at={(obj.east)},
  xshift=1.2cm,
  anchor=west,
  width=0.47\textwidth,
  height=0.42\textwidth,
  xmin=0,xmax=1,ymin=0,ymax=1.05,
  axis lines=left,
  xlabel={$\lambda$},
  ylabel={$t$},
  grid=both,
  tick label style={font=\scriptsize},
  label style={font=\small},
  title={Lifted HV: Tchebycheff shadow region},
  title style={font=\small},
  clip=false
]
\addplot[name path=ref, thick, domain=0:1, samples=200] {max(x,1-x)};
\addplot[name path=env, very thick, domain=0:1, samples=300]
{min(min(max(0.12*x,0.88*(1-x)),max(0.30*x,0.55*(1-x))),min(max(0.55*x,0.32*(1-x)),max(0.84*x,0.16*(1-x))))};
\addplot[fill=blue!18, draw=none] fill between[of=ref and env, soft clip={domain=0:1}];
\addplot[thick, domain=0:1, samples=200] {max(x,1-x)};
\addplot[very thick, domain=0:1, samples=300]
{min(min(max(0.12*x,0.88*(1-x)),max(0.30*x,0.55*(1-x))),min(max(0.55*x,0.32*(1-x)),max(0.84*x,0.16*(1-x))))};
\node[anchor=south,font=\scriptsize] at (axis cs:0.22,0.78) {$h_r(\lambda)$};
\node[anchor=north,font=\scriptsize] at (axis cs:0.62,0.25) {$h_A(\lambda)$};
\node[font=\scriptsize] at (axis cs:0.50,0.42) {TSM area};
\end{axis}
\end{tikzpicture}
\caption{Two-dimensional mapping from Tchebycheff $R_2$ to a lifted hypervolume-like shadow region. Left: four mutually nondominated objective-space points. Right: the reference envelope $h_r$ and the attained envelope $h_A$ in scalarization space. The shaded region is the Tchebycheff shadow magnitude and equals exact integral $R_2$ improvement.}
\label{fig:tsm-2d-map}
\end{figure}
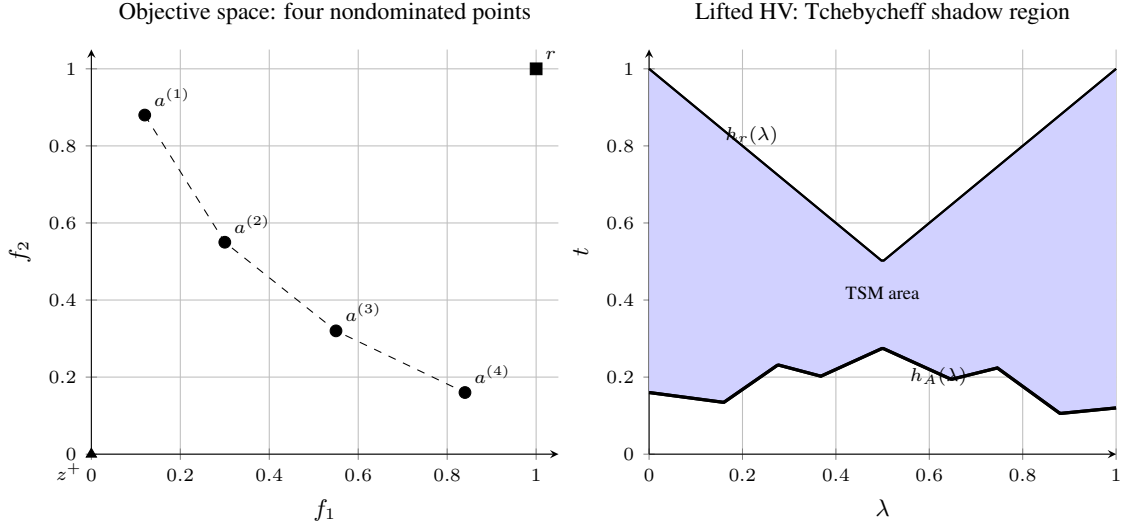

\section{Canonical EHVI and the PHVI representation}
\label{sec:phvi}

We next recall canonical EHVI and state the representation that serves as the starting point for the hypervolume part of the paper. The main point is that EHVI is not merely analogous to hypervolume improvement; under independent Gaussian predictive coordinates it can be written as an ordinary hypervolume improvement after a deterministic coordinate transformation. This observation justifies transferring known properties of HVI to EHVI with some care.

Canonical EHVI is the expected value of deterministic hypervolume improvement under a predictive distribution.

\begin{theorem}[PHVI representation of EHVI]
For independent Gaussian predictive coordinates and a rectangular reference orientation, EHVI can be written as an ordinary hypervolume improvement after applying coordinate-wise expected-improvement transformations. In particular, in the normalized maximization case with reference zero,
\[
\tilde r_j=\EI(0;\mu_j,\sigma_j^2),\qquad
\tilde a_j^{(i)}=\tilde r_j-\EI(a_j^{(i)};\mu_j,\sigma_j^2),
\]
and
\[
\EHVI_A(\mu,\sigma)=\HVI(\{\tilde r\},\tilde A).
\]
\end{theorem}

The theorem should be read as a change of geometry. The expectation over a Gaussian predictive distribution is absorbed into the coordinate transformation. Once this has been done, the remaining object is an ordinary deterministic hypervolume improvement. This is why the theorem is useful not only for computation, but also for structural arguments.

\begin{theorem}[Canonical EHVI is mean- and variance-monotone]
Assume independent Gaussian predictive coordinates and a fixed incumbent set. In minimization orientation, improving the predictive mean coordinatewise cannot decrease EHVI, and increasing any predictive standard deviation while keeping the mean fixed cannot decrease EHVI.
\end{theorem}

\section{Preference-shaped EHVI}
\label{sec:preference-ehvi}

Three preference mechanisms for EHVI are considered here. Product densities change the measure used by hypervolume. Cone orders change the dominance relation. Truncation changes the predictive distribution by conditioning on a region of interest. These operations are sometimes grouped together under preference handling, but mathematically they have different consequences. We therefore treat them separately.

\subsection{Positive product-density and desirability transformations}

Product densities shape hypervolume by changing the local importance of objective-space regions. Desirability functions provide a convenient coordinate-wise interpretation of such densities.

A product density has the form
\[
K(y)=\prod_{j=1}^m k_j(y_j),\qquad k_j(t)=D_j'(t),
\]
where $D_j$ is a coordinate-wise desirability index. Such density transformations are closely related to reference-point-free weighted hypervolume and desirability-function interpretations \cite{EmmerichDeutzYevseyeva2014}, and to the D-PHI aspiration--reservation construction \cite{LiangShavazipourSainiEmmerich2026}. The transformation
\[
T_j(t)=D_j(t)=\int^t k_j(s)\dd s
\]
maps weighted hypervolume into ordinary hypervolume in desirability space. Figure~\ref{fig:dphi-desirability} shows this construction for a single objective: aspiration and reservation levels define a monotone desirability map, and its derivative supplies the marginal kernel factor that weights hypervolume improvement.

\begin{theorem}[Positive product-density EHVI]
Let $K(y)=\prod_j k_j(y_j)$, where each $k_j\in L^1$ is nonnegative and positive almost everywhere on the relevant interval. Then the corresponding weighted hypervolume is an ordinary hypervolume after the coordinate-wise transformation $T=(T_1,\ldots,T_m)$. If the predictive model is Gaussian in transformed coordinates, expected weighted hypervolume improvement inherits exact computation, strict Pareto compliance, and variance monotonicity from canonical EHVI.
\end{theorem}

\begin{proposition}[Positive kernels for strict Pareto compliance]
A weighted hypervolume density $K$ yields a finite strictly Pareto-compliant weighted hypervolume on the reference-bounded domain $\Omega_r$ if and only if $K\in L^1(\Omega_r)$ and $K>0$ almost everywhere on $\Omega_r$. In the product-density case $K(y)=\prod_j k_j(y_j)$ with nonnegative marginal kernels, this is equivalent to $k_j>0$ almost everywhere for every coordinate $j$.
\end{proposition}

\subsection{Cone-based EHVI}

Cone-based hypervolume changes the dominance order itself. This subsection explains why simplicial cone orders can still be handled by ordinary hypervolume after a linear coordinate transformation.

Cone orders were introduced for hypervolume indicators to encode admissible trade-off directions and maximum trade-off constraints \cite{EmmerichDeutzKruisselbrinkShukla2013}. Figure~\ref{fig:cone-hv-2d} compares the ordinary Pareto cone with a wider cone and shows how the dominated region changes when the preference order is changed.

\begin{proposition}[Cone hypervolume reduces to ordinary hypervolume]
For a simplicial cone $C$ and $L=C^{-1}$,
\[
\CHV_{C,r}(A)=|\det C|\,\HV_{Lr}(LA).
\]
Consequently, cone-based EHVI is ordinary EHVI in cone coordinates, up to a positive determinant factor.
\end{proposition}

\subsection{Truncated EHVI}

Truncated EHVI focuses the predictive distribution on a region of interest. This is useful for preference focusing, but its monotonicity behaviour differs from product-density and cone transformations.

Truncated EHVI was proposed as a region-of-interest acquisition for preference-based Bayesian multiobjective optimization \cite{YangLiDeutzBackEmmerich2016}. A region of interest may be specified by a box $[a,b]$. Truncated EHVI uses the Gaussian predictive density conditioned on $Y\in[a,b]$:
\[
\TEHVI_A(\mu,\sigma;a,b)=\int_{[a,b]}I_A(y)p_{\mu,\sigma}^{[a,b]}(y)\dd y.
\]

\begin{proposition}[Truncated EHVI]
For independent truncated coordinates, a PHVI-type representation remains available. Mean monotonicity holds for fixed truncation bounds, but variance monotonicity can fail in general for fixed-box truncation.
\end{proposition}

\section{Exact \texorpdfstring{$R_2$}{R2} improvement and scalarization-space volume}
\label{sec:r2-volume}

We now move from hypervolume to \(R_2\). The central distinction is that \(R_2\) does not measure an objective-space dominated region. It measures the quality of a set through the lower envelope of Tchebycheff scalarization values. For the improvement form of exact integral \(R_2\), this gives a volume between two scalarization envelopes in the space of weights and achievement levels.

\begin{definition}[Scalarization-space shadow]
Define
\[
\mathcal H(A;r)=\{(\lambda,t):\lambda\in\DeltaM,\ h_A(\lambda)\le t\le h_r(\lambda)\}.
\]
\end{definition}

\begin{theorem}[Exact $R_2$ improvement as scalarization-space volume]
For every finite approximation set $A$,
\[
I_{R_2}(A;r)=\int_{\mathcal H(A;r)}\rho(\lambda)\dd\sigma(\lambda)\dd t.
\]
\end{theorem}

The result is elementary, but important. It says that exact \(R_2\) improvement is an ordinary volume only after lifting the problem to \((\lambda,t)\)-space. The variable \(\lambda\) records the scalarization weight, while \(t\) records the achieved scalarization level. Thus the integral is not over objective space; it is over the space in which \(R_2\) is naturally defined. The geometric transition is the one shown in Figures~\ref{fig:r2-objective-view},~\ref{fig:r2-weight-view}, and~\ref{fig:tsm-2d-map}: weighted Tchebycheff rectangles touch the front in objective space, their best values form an envelope over weights, and the area between the incumbent and reference envelopes is the scalarization-space improvement volume.

\begin{theorem}[No general objective-space weighted-HV representation]
There is no Lebesgue-integrable objective-space density kernel $K$ such that $\WHV_{K,r}(A)=I_{R_2}(A;r)$ for all finite sets $A$, already in two-objective minimization with $z^+=(0,0)$ and $r=(1,1)$.
\end{theorem}

This negative result prevents a misleading simplification. Exact integral \(R_2\) is not merely hypervolume with a different density. The proof shows that the two constructions disagree already on simple boundary examples: objective-space hypervolume sees zero area, whereas Tchebycheff scalarization still sees improvement for a set of weights of positive measure.

\section{Magnitude, Tchebycheff shadows, and ER2I}
\label{sec:magnitude-shadows}

The role of lower-dimensional contributions is addressed next. The negative result for objective-space weighted hypervolume suggests that one should look at indicators which can see more than full-dimensional volume. Metric magnitude is relevant for precisely this reason. Nevertheless, exact \(R_2\) is organized by scalarization directions rather than coordinate projections; this leads to the Tchebycheff shadow representation.

\begin{proposition}[Magnitude removes the zero-volume boundary obstruction]
There exist singleton approximation sets with zero ordinary hypervolume contribution but positive reduced magnitude contribution. For example, with $r=(1,1)$ and $A_c=\{(1,c)\}$, $0<c<1$, the dominated set has zero area but positive one-dimensional projection length.
\end{proposition}

\begin{theorem}[No clean bijection from standard magnitude to exact integral $R_2$]
In two-objective minimization with $z^+=(0,0)$, $r=(1,1)$, and uniform weights, exact integral Tchebycheff $R_2$ improvement is not a function of standard reduced magnitude of the dominated set.
\end{theorem}

\begin{definition}[Tchebycheff shadow representation]
For $D=\Dom_r(A)$ define
\[
h_D(\lambda)=\inf_{y\in D}g_\lambda(y;z^+)=h_A(\lambda),
\]
and the one-dimensional shadow
\[
\operatorname{Sh}_\lambda(D;r)=\bigl[h_D(\lambda),h_r(\lambda)\bigr],
\]
empty if the lower endpoint exceeds the upper endpoint. The Tchebycheff shadow magnitude is
\[
\TSM_\rho(D;r)=\int_{\DeltaM}\mathcal H^1(\operatorname{Sh}_\lambda(D;r))\rho(\lambda)\dd\sigma(\lambda).
\]
\end{definition}

\begin{theorem}[Tchebycheff shadow representation]
For every finite approximation set $A$ and $D=\Dom_r(A)$,
\[
\TSM_\rho(D;r)=I_{R_2}(A;r).
\]
Moreover, if $\Dom_r(B)\subseteq\Dom_r(A)$, then
\[
\TSM_\rho(\Dom_r(A);r)\ge \TSM_\rho(\Dom_r(B);r).
\]
\end{theorem}

The term ``shadow'' is geometric language for an ordinary integral in weight-achievement space. For each weight, the current approximation set casts an interval between its envelope value and the reference envelope. Integrating the lengths of these intervals gives exact \(R_2\) improvement. The terminology is useful because it prevents confusing this object with objective-space hypervolume.

\section{ER2I and achievement-space Gaussian surrogates}
\label{sec:er2i-surrogates}

The deterministic scalarization-shadow representation now becomes an acquisition function. The decisive modelling choice is whether one starts with a Gaussian model for the objective vector or learns Gaussian marginals directly for the scalar achievement process. The latter choice gives a particularly simple ER2I formula; the former is more faithful to a joint objective model but leads to a more complicated integrand.

Let $Y(x)$ denote the random objective vector at a candidate design $x$. Adding $Y(x)$ improves the scalarization envelope by
\[
\Delta_\lambda(Y(x);A)=\bigl(h_A(\lambda)-g_\lambda(Y(x);z^+)\bigr)_+.
\]
Hence
\[
\ERII_A(x)=\int_{\DeltaM}\E[\Delta_\lambda(Y(x);A)]\rho(\lambda)\dd\sigma(\lambda).
\]

\begin{theorem}[Achievement-space Gaussian ER2I]
Assume that the scalarized achievement process is modelled directly by Gaussian marginals
\[
Z_x(\lambda)\sim N(m_x(\lambda),s_x^2(\lambda)).
\]
Then
\[
\ERII_A(x)=\int_{\DeltaM}\EI\bigl(h_A(\lambda);m_x(\lambda),s_x(\lambda)\bigr)\rho(\lambda)\dd\sigma(\lambda),
\]
where
\[
\EI(c;\mu,s)=(c-\mu)\Phi\!\left(\frac{c-\mu}{s}\right)+s\varphi\!\left(\frac{c-\mu}{s}\right).
\]
\end{theorem}

This formula is the computational payoff of the scalarization-space viewpoint. If the achievement process is modelled directly, the integrand is the classical one-dimensional Gaussian expected improvement. The remaining difficulty is integration over weights, not hypervolume decomposition in objective space.

\begin{corollary}[Variance monotonicity in achievement space]
For Gaussian achievement marginals, finite-weight ER2I and quadrature/integral ER2I are nondecreasing in every achievement standard deviation $s_x(\lambda)$.
\end{corollary}

\begin{proposition}[Independent objective-Gaussian integrand]
If $Y_i(x)\sim N(\mu_i(x),\sigma_i^2(x))$ independently, then for interior weights
\[
\E[\Delta_\lambda(Y(x);A)]
=
\int_{-\infty}^{h_A(\lambda)}\prod_{i=1}^m
\Phi\!\left(\frac{z_i^+ + t/\lambda_i-\mu_i(x)}{\sigma_i(x)}\right)\dd t.
\]
Variance monotonicity is then conditional and can fail when the threshold lies above the scaled objective mean.
\end{proposition}

\section{Algorithms and numerical computation}
\label{sec:algorithms}

The representation theorems now translate into algorithmic form. No claim is made that exact \(R_2\) becomes an ordinary hypervolume computation. Rather, the Tchebycheff shadow representation identifies the correct integration domain. For finite weights this gives a finite sum; for exact integral \(R_2\) it gives numerical integration over the simplex.

For a discrete weight set $\Lambda_K=\{\lambda^{(1)},\ldots,\lambda^{(K)}\}$, train one scalar GP for each achievement function $x\mapsto g_{\lambda^{(k)}}(f(x);z^+)$. The current envelope is
\[
h_{n,k}=\min_i g_{\lambda^{(k)}}(f(x_i);z^+).
\]
The acquisition is the finite sum
\[
\ERII_{K,n}(x)=\frac1K\sum_{k=1}^K\EI\bigl(h_{n,k};m_k(x),s_k(x)\bigr).
\]

\begin{proposition}[Discrete achievement-space ER2I algorithm]
A sum-based discrete $R_2$ acquisition can be implemented by updating $K$ scalar envelope values, training $K$ scalar achievement surrogates, and maximizing the average of $K$ scalar Gaussian expected improvements.
\end{proposition}

A schematic pseudocode is:
\begin{enumerate}[label=\textbf{Step \arabic*.},leftmargin=2.3cm]
\item Choose weights $\Lambda_K$ and initial designs $x_1,\ldots,x_{n_0}$.
\item Evaluate $f(x_i)=(f_1(x_i),f_2(x_i))$.
\item For each weight, compute achievement targets $z_{ik}=g_{\lambda^{(k)}}(f(x_i);z^+)$ and the envelope $h_{n,k}=\min_i z_{ik}$.
\item Train scalar GP$_k$ on $(x_i,z_{ik})$ for every $k$.
\item Maximize $\ERII_{K,n}(x)$ over the design domain.
\item Evaluate the selected design, update the envelopes, and repeat.
\end{enumerate}

For exact integral $R_2$, replace the finite sum by quadrature on the simplex.

\begin{theorem}[Numerical integration form of exact ER2I]
Let $Q_L=\{(\lambda^{(\ell)},w_\ell)\}_{\ell=1}^L$ be a quadrature rule on $\DeltaM$. Under the Gaussian achievement surrogate,
\[
\widehat{\ERII}_{A,L}(x)=\sum_{\ell=1}^L w_\ell\,\EI\bigl(h_A(\lambda^{(\ell)});m_x(\lambda^{(\ell)}),s_x(\lambda^{(\ell)})\bigr)\rho(\lambda^{(\ell)})
\]
is a quadrature approximation of exact integral ER2I.
\end{theorem}

The cost of updating all envelope values after a new point is $O(Lm)$. One acquisition evaluation costs $O(L C_{\rm pred}+L)$, where $C_{\rm pred}$ is the cost of one predictive mean--variance query. Naively training one GP per quadrature node costs $O(Ln^3)$, but sparse GPs, shared kernels, multitask models, and adaptive quadrature can reduce the cost.

\section{Conclusions}
\label{sec:conclusions}
We return, finally, to the three guiding modes: computation, compatibility, and monotonicity. Table~\ref{tab:final} summarizes the results according to these modes. 

The paper has compared preference-shaped expected improvement criteria on the hypervolume and \(R_2\) sides. The main lesson is that the correct geometry matters, and this geometry should be identified before one designs algorithms. EHVI lives naturally in objective-space hypervolume geometry, and many preference transformations can be understood by asking whether they preserve this geometry after a coordinate or cone transformation. Exact integral \(R_2\), in contrast, lives naturally in scalarization-envelope geometry. Treating exact \(R_2\) as ordinary objective-space hypervolume obscures both the obstruction and the computational opportunity.

The table also indicates where the novelty lies. On the EHVI side, the novel synthesis is that product-density transformations, cone transformations, and truncation can be compared by the same three guiding modes. Product densities and cones preserve a strong link to ordinary hypervolume geometry, whereas truncation changes the uncertainty response. On the \(R_2\) side, the negative result rules out ordinary objective-space weighted hypervolume as a general representation, while the positive Tchebycheff-shadow theorem provides the correct scalarization-space representation. This in turn yields a practical computation principle for ER2I: finite sums for discrete \(R_2\) and quadrature over the weight simplex for exact integral \(R_2\).

Several directions remain open. They are mathematical as well as computational. The first is to develop efficient exact or adaptive algorithms for the simplex subdivisions induced by \(h_A(\lambda)\), especially in three objectives. The second is to study statistically consistent multi-task surrogate models in weight-achievement space, instead of independent scalar GPs per weight. The third is to connect Tchebycheff shadows more systematically with magnitude-type mixed-dimensional valuations. These directions all preserve the central distinction of the paper: hypervolume and \(R_2\) are related expected-indicator principles, but their natural geometries are different.

A further line of future work is to integrate these criteria more fully into Bayesian optimization workflows and to study their empirical performance under comparable surrogate models, preference specifications, and benchmark settings. First steps in this direction have already been taken by the work of Yang et al. on truncated expected hypervolume improvement for region-of-interest Bayesian multiobjective optimization \cite{YangLiDeutzBackEmmerich2016}, and by the D-PHI framework of Liang et al., which develops desirability-based hypervolume criteria from aspiration and reservation levels \cite{LiangShavazipourSainiEmmerich2026}. The latter translates to a weighted kernel for the hypervolume expected improvement indicator and uses the D-PHI indicator in Bayesian frameworks. Systematic comparisons of such criteria, including their computational cost, robustness to preference misspecification, and behavior in higher-dimensional objective spaces, remain an important next step.

\noindent
\textbf{Reproducibility}\\
For reproducibility and implementations, see
\begin{center}
 \href{https://github.com/emmerichmtm}{https://github.com/emmerichmtm}
\end{center}

\noindent
\textbf{Acknowledgements:}\\
This research is related to the thematic research area Decision Analytics utilizing Causal Models and Multiobjective Optimization (DEMO, jyu.fi/demo) at the University of Jyvaskyla, which was partly funded by the Academy of Finland (project 322221). 

\begin{table}[t]
\centering
\small
\setlength{\tabcolsep}{3pt}
\renewcommand{\arraystretch}{1.18}
\begin{tabular}{@{}p{3.0cm}p{4.2cm}p{3.6cm}p{3.3cm}@{}}
\toprule
Criterion & Computation & Pareto compatibility & Variance monotonicity\\
\midrule
Canonical EHVI & PHVI representation as ordinary HVI; exact algorithms transfer & yes & \textbf{dimension-free via convex order}\\
Product-density EHVI & \textbf{ordinary HVI after desirability transform} & \textbf{positive-a.e. kernels give strict compliance} & \textbf{preserved in transformed Gaussian coordinates}\\
Cone EHVI & \textbf{ordinary EHVI in cone coordinates} & \textbf{cone-order compatibility} & \textbf{variance monotone in cone coordinates}\\
Truncated EHVI & truncated-coordinate PHVI representation & fixed region-of-interest compatibility & \textbf{mean monotone but variance monotonicity may fail}\\
Discrete ER2I & \textbf{one scalar achievement GP per weight} & weight dependent; inherited from discrete \(R_2\) & \textbf{monotone for achievement-Gaussian marginals}\\
Exact ER2I & \textbf{Tchebycheff-shadow quadrature on \(\DeltaM\)} & \textbf{Pareto-compliant for integrated \(R_2\)} & \textbf{conditional for objective-Gaussian models}\\
\bottomrule
\end{tabular}
\caption{Summary of the paper's results by guiding mode. Boldface marks new theorem-level results, negative results, or unifying formulations emphasized in this paper.}
\label{tab:final}
\end{table}

\appendix
\section{List of symbols and terms}
\label{app:symbols}

\begin{center}
\small
\begin{tabular}{@{}p{0.22\textwidth}p{0.70\textwidth}@{}}
\toprule
Symbol or term & Meaning in this paper\\
\midrule
\(m\) & Number of objectives.\\
\(x\) & Design or decision vector evaluated by the black-box objectives.\\
\(f(x)=(f_1(x),\ldots,f_m(x))\) & Objective vector, written in minimization orientation unless stated otherwise.\\
\(Y(x)\) & Predictive random objective vector at candidate \(x\).\\
\(A\) & Current finite approximation set in objective space.\\
\(r\) & Dystopian or nadir-type reference point for hypervolume.\\
\(z^+\) & Utopian or ideal reference point for Tchebycheff scalarizations and \(R_2\).\\
\(\Dom_r(A)\) & Reference-bounded dominated region generated by \(A\).\\
\(\HV_r(A)\) & Hypervolume of \(A\) with respect to \(r\).\\
\(\HVI_r(\{y\},A)\) & Hypervolume improvement of a candidate objective vector \(y\).\\
\(\EHVI_A\) & Expected hypervolume improvement under the predictive distribution.\\
\(K(y)\) & Weighted hypervolume density kernel.\\
\(\WHV_{K,r}(A)\) & Weighted hypervolume of \(A\) with density \(K\).\\
\(D_j\), \(T_j\) & Coordinate-wise desirability or product-kernel transformation.\\
\(C\) & Cone defining a cone order; for simplicial cones, \(L=C^{-1}\).\\
\(\CHV_{C,r}(A)\) & Cone-based hypervolume.\\
\(\TEHVI\) & Truncated expected hypervolume improvement in a region of interest.\\
\(\Delta^{m-1}\) & Unit weight simplex \(\{\lambda\in\R_+^m:\sum_i\lambda_i=1\}\).\\
\(\lambda\) & Weight vector for Tchebycheff scalarization.\\
\(\rho(\lambda)\) & Nonnegative weight density on the simplex.\\
\(g_\lambda(y;z^+)\) & Weighted Tchebycheff achievement value \(\max_i\lambda_i(y_i-z_i^+)\).\\
\(h_A(\lambda)\) & Scalarization envelope \(\min_{a\in A}g_\lambda(a;z^+)\).\\
\(R_2(A)\) & Exact integral \(R_2\) value \(\int h_A(\lambda)\rho(\lambda)\dd\sigma(\lambda)\).\\
\(I_{R_2}(A;r)\) & Deterministic \(R_2\)-improvement relative to \(r\).\\
\(\ERII_A(x)\) & Expected \(R_2\)-indicator improvement of candidate \(x\).\\
\(\mathcal H(A;r)\) & Scalarization-space shadow region between \(h_A\) and \(h_r\).\\
\(\operatorname{Sh}_\lambda(D;r)\) & Tchebycheff shadow interval at weight \(\lambda\).\\
\(\TSM_\rho(D;r)\) & Tchebycheff shadow representation of exact integral \(R_2\) improvement.\\
\(\EI(c;\mu,s)\) & Scalar Gaussian expected improvement at threshold \(c\), mean \(\mu\), and standard deviation \(s\).\\
\(\Phi,\varphi\) & Standard normal distribution function and density.\\
\bottomrule
\end{tabular}
\end{center}

\paragraph{Terminology.}
\emph{Objective-space hypervolume} refers to volume of dominated regions in the original objective coordinates. \emph{Scalarization-space volume} refers to integration over \((\lambda,t)\), where \(\lambda\) is a weight and \(t\) is an achievement level. \emph{Tchebycheff shadow} is the geometric interval between the current scalarization envelope and the reference envelope. It is analytically an ordinary integral in weight-achievement space, but conceptually it is useful because it distinguishes exact \(R_2\) improvement from objective-space hypervolume.

\section{Proofs for canonical and preference-shaped EHVI}
\label{app:ehvi-proofs}

\begin{proof}[Proof sketch of the PHVI representation]
For independent predictive coordinates, the expected measure of a rectangular contribution factorizes into a product of one-dimensional expected-improvement terms. The deterministic HVI contribution can be decomposed into signed unions of axis-aligned boxes. Applying Fubini's theorem to each box replaces each coordinate interval by a one-dimensional Gaussian tail integral. These tail integrals are precisely the coordinate-wise EI transforms
\[
\tilde r_j=\EI(0;\mu_j,\sigma_j^2),
\qquad
\tilde a_j^{(i)}=\tilde r_j-\EI(a_j^{(i)};\mu_j,\sigma_j^2).
\]
The inclusion--exclusion pattern of the boxes is unchanged; only the coordinates are transformed. Hence the original expected HVI equals the ordinary HVI of the transformed point set and transformed reference point.
\end{proof}

\begin{proof}[Proof of canonical EHVI mean and variance monotonicity]
Let $I_A(y)$ denote deterministic hypervolume improvement. For ordinary Pareto hypervolume it is coordinate-wise monotone in the improvement direction and convex in each candidate coordinate on every cell of the orthogonal decomposition induced by the incumbent set. The PHVI representation transfers these properties to the transformed coordinates. Improving the predictive mean shifts the one-dimensional EI transforms in the favourable direction, so the ordinary HVI in transformed space cannot decrease.

For variance monotonicity, fix the mean and increase one predictive standard deviation. A centered normal distribution with larger variance dominates one with smaller variance in the convex order. Since the one-dimensional contribution entering the HVI representation is a convex positive-part expectation, the transformed coordinate is nondecreasing in the predictive standard deviation. Monotonicity of ordinary HVI in the transformed coordinates then implies nondecreasing EHVI. Repeating this argument coordinate by coordinate gives the dimension-free statement.
\end{proof}

\begin{proof}[Proof of the positive product-density theorem]
Assume $K(y)=\prod_{j=1}^m k_j(y_j)$ with $k_j\ge0$, $k_j\in L^1$ on the relevant coordinate interval, and $k_j>0$ almost everywhere. Define
\[
T_j(t)=\int_{\ell_j}^{t}k_j(s)\dd s,\qquad T(y)=(T_1(y_1),\ldots,T_m(y_m)).
\]
The map $T$ is coordinate-wise nondecreasing and strictly increasing in the measure-theoretic sense: every interval of positive length is mapped to an interval of positive length. For every dominated box $B=\prod_j[a_j,r_j]$,
\[
\int_B K(y)\dd y
=\prod_{j=1}^m\int_{a_j}^{r_j}k_j(t)\dd t
=\prod_{j=1}^m\bigl(T_j(r_j)-T_j(a_j)\bigr),
\]
which is the ordinary hypervolume of the transformed box. By finite additivity over the standard box decomposition of $\Dom_r(A)$, the weighted hypervolume of $A$ equals ordinary hypervolume of $T(A)$ with reference $T(r)$. Therefore expected weighted hypervolume improvement is ordinary EHVI in transformed coordinates. Strict Pareto compliance follows because a positive-volume improvement region has positive $K$-measure when $K>0$ almost everywhere; variance monotonicity follows by applying the canonical EHVI result to a Gaussian model posed in transformed coordinates.
\end{proof}

\begin{proof}[Proof of the cone hypervolume proposition]
Let $C$ be a simplicial cone generated by the columns of an invertible matrix $C$ and let $L=C^{-1}$. By definition, $y'$ is dominated from $y$ under the cone order precisely when $y'-y\in C\R_+^m$. Applying $L$ gives
\[
L(y'-y)\in\R_+^m,
\]
which is the ordinary orthant order in transformed coordinates. Thus the cone-dominated region generated by $A$ and bounded by $r$ is mapped by $L$ to the ordinary dominated region generated by $LA$ and bounded by $Lr$. The change-of-variables formula gives the volume factor $|\det C|$, because $|\det L|=1/|\det C|$ depending on the direction in which the region is transformed. With the convention used in the statement,
\[
\CHV_{C,r}(A)=|\det C|\,\HV_{Lr}(LA).
\]
Expected improvement follows by applying the same deterministic transformation inside the predictive expectation.
\end{proof}

\begin{proof}[Proof of truncated-EHVI claims]
For independent truncated coordinates, the conditional density on a fixed box $[a,b]$ factorizes into one-dimensional truncated Gaussian densities. The same Fubini argument as in the PHVI representation therefore applies, with one-dimensional truncated EI transforms replacing ordinary Gaussian EI transforms. This yields a PHVI-type representation in truncated coordinates.

For fixed truncation bounds, a favourable shift of the predictive mean improves the one-dimensional truncated EI transform in the improvement direction, so mean monotonicity is preserved. Variance monotonicity is different: the conditional distribution is renormalized on a fixed box. Increasing variance may move probability mass away from the part of the box that contributes improvement and may reduce the normalized density in the useful region. A one-dimensional counterexample embedded as a coordinate slice of the multiobjective problem therefore shows that fixed-box TEHVI is not variance-monotone in general.
\end{proof}

\section{Proof of the weighted-kernel admissibility proposition}
\label{app:kernel-proof}

\begin{proof}
If $K\in L^1(\Omega_r)$ and $K>0$ almost everywhere, then $\WHV_{K,r}(A)$ is finite because $\Dom_r(A)\subseteq\Omega_r$. If $\Dom_r(B)\subsetneq\Dom_r(A)$ and the improvement region
\[
S=\Dom_r(A)\setminus\Dom_r(B)
\]
has positive Lebesgue measure, then
\[
\WHV_{K,r}(A)-\WHV_{K,r}(B)
=
\int_S K(y)\dd y>0,
\]
because a nonnegative integrable function that is positive almost everywhere has positive integral on every measurable set of positive measure. Hence the weighted hypervolume is strictly Pareto-compliant in the reference-bounded domain.

Conversely, strict Pareto compliance first implies ordinary Pareto monotonicity. If $K<0$ on a set of positive measure, regularity of the Lebesgue integral and the differentiation theorem give an axis-aligned box $Q=\prod_i[\alpha_i,\beta_i]\subseteq\Omega_r$ with $\int_QK(y)\dd y<0$. Let $a=(\alpha_1,\ldots,\alpha_m)$ and
\[
b^{(i)}=(\alpha_1,\ldots,\alpha_{i-1},\beta_i,\alpha_{i+1},\ldots,\alpha_m),
\quad i=1,\ldots,m.
\]
Set $A=\{a\}$ and $B=\{b^{(1)},\ldots,b^{(m)}\}$. Then $\Dom_r(B)\subseteq\Dom_r(A)$ and, up to a null boundary,
\[
\Dom_r(A)\setminus\Dom_r(B)=Q.
\]
Thus
\[
\WHV_{K,r}(A)-\WHV_{K,r}(B)=\int_QK(y)\dd y<0,
\]
contradicting even ordinary Pareto monotonicity. Therefore $K\ge0$ almost everywhere. If $K=0$ on a measurable set $E\subseteq\Omega_r$ of positive measure, then an admissible positive-measure improvement region contained in $E$ receives zero additional value,
\[
\int_E K(y)\dd y=0,
\]
contradicting strict Pareto compliance in the measure-theoretic sense stated in the proposition. Hence $K>0$ almost everywhere. Finiteness on the full reference box gives $K\in L^1(\Omega_r)$.

For a product density $K(y)=\prod_j k_j(y_j)$ with nonnegative marginals, $K>0$ almost everywhere on the product domain if and only if every factor $k_j$ is positive almost everywhere on its coordinate interval. This gives the marginal-kernel condition.
\end{proof}

\section{Proofs for exact \texorpdfstring{$R_2$}{R2}, magnitude, and Tchebycheff shadows}

\begin{proof}[Proof of scalarization-space volume theorem]
By the layer-cake identity,
\[
(h_r(\lambda)-h_A(\lambda))_+
=
\int_0^{h_r(\lambda)}\1\{h_A(\lambda)\le t\}\dd t.
\]
Multiplying by $\rho(\lambda)$ and integrating over $\DeltaM$ gives
\[
I_{R_2}(A;r)
=
\int_{\DeltaM}\int_0^{h_r(\lambda)}
\1\{h_A(\lambda)\le t\}\rho(\lambda)\dd t\dd\sigma(\lambda),
\]
which is exactly the weighted measure of
\[
\mathcal H(A;r)=\{(\lambda,t):h_A(\lambda)\le t\le h_r(\lambda)\}.
\]
\end{proof}

\begin{proof}[Proof of the no-objective-space-WHV theorem]
Consider two objectives, $z^+=(0,0)$, $r=(1,1)$, and $A_c=\{(1,c)\}$ with $0<c<1$. The ordinary dominated region is
\[
\Dom_r(A_c)=[1,1]\times[c,1],
\]
which has two-dimensional Lebesgue measure zero. Hence every integrable objective-space density gives $\WHV_{K,r}(A_c)=0$.

For $\lambda\in[0,1]$,
\[
g_\lambda(r;0)=\max\{\lambda,1-\lambda\},
\qquad
g_\lambda((1,c);0)=\max\{\lambda,(1-\lambda)c\}.
\]
On $0\le\lambda\le c/(1+c)$ we have $g_\lambda((1,c);0)=(1-\lambda)c$ and $g_\lambda(r;0)=1-\lambda$. Hence
\[
g_\lambda(r;0)-g_\lambda((1,c);0)=(1-\lambda)(1-c)>0
\]
on a set of positive measure. Therefore $I_{R_2}(A_c;r)>0$, contradicting any equality with ordinary objective-space weighted hypervolume valid for all finite sets.
\end{proof}

\begin{proof}[Proof that magnitude removes the zero-volume obstruction]
With $r=(1,1)$ and $A_c=\{(1,c)\}$, $0<c<1$, the dominated region is the vertical segment
\[
D=[1,1]\times[c,1].
\]
Its two-dimensional area is zero, so ordinary hypervolume is zero. In the reduced magnitude expansion for rectangular sets in two dimensions, the one-dimensional projection terms contribute. The projection of $D$ onto the first coordinate has length zero, whereas the projection onto the second coordinate has length $1-c$. Thus the reduced magnitude contribution is positive, equal to $(1-c)/2$ under the standard two-dimensional normalization. Hence magnitude detects a lower-dimensional contribution that ordinary hypervolume misses.
\end{proof}

\begin{proof}[Proof of the no-magnitude-bijection theorem]
In two objectives with $z^+=(0,0)$, $r=(1,1)$, and uniform weights, take
\[
A=\{(0,1)\}.
\]
Its reduced magnitude is $1/2$, and its exact $R_2$ improvement is $1/4$. Now let $B=\{(s,s)\}$ with $s=3-\sqrt 6$ and write $p=1-s=\sqrt6-2$. The dominated box of $B$ has side lengths $p,p$, so its reduced magnitude is
\[
p+\frac{p^2}{4}=\frac12.
\]
Thus $A$ and $B$ have the same reduced magnitude. However,
\[
g_\lambda((s,s);0)=s\,g_\lambda(r;0),
\]
and therefore
\[
I_{R_2}(B;r)
=
(1-s)\int_0^1\max\{\lambda,1-\lambda\}\dd\lambda
=
(\sqrt6-2)\frac34
\neq \frac14=I_{R_2}(A;r).
\]
Thus exact integral $R_2$ improvement is not a function of standard reduced magnitude alone.
\end{proof}

\begin{proof}[Proof of the Tchebycheff shadow theorem]
For $D=\Dom_r(A)$, monotonicity of $g_\lambda$ implies
\[
h_D(\lambda)=\inf_{y\in D}g_\lambda(y;z^+)=\min_{a\in A}g_\lambda(a;z^+)=h_A(\lambda).
\]
Therefore
\[
\mathcal H^1(\operatorname{Sh}_\lambda(D;r))=(h_r(\lambda)-h_A(\lambda))_+.
\]
Integrating over $\lambda$ gives $\TSM_\rho(D;r)=I_{R_2}(A;r)$. If $\Dom_r(B)\subseteq\Dom_r(A)$, then $h_{\Dom_r(A)}(\lambda)\le h_{\Dom_r(B)}(\lambda)$ for all $\lambda$, proving Pareto monotonicity after integration with the nonnegative density $\rho$.
\end{proof}

\section{Proofs for ER2I formulas}

\begin{proof}[Proof of achievement-space Gaussian ER2I]
For fixed $\lambda$, adding candidate $x$ changes the envelope to $\min\{h_A(\lambda),Z_x(\lambda)\}$, so the improvement is $(h_A(\lambda)-Z_x(\lambda))_+$. If $Z_x(\lambda)$ is Gaussian with mean $m_x(\lambda)$ and standard deviation $s_x(\lambda)$, its expectation is the scalar Gaussian EI formula. Integration over $\DeltaM$ gives the theorem.
\end{proof}

\begin{proof}[Proof of variance monotonicity in achievement space]
For $s>0$,
\[
\EI(c;\mu,s)=(c-\mu)\Phi(u)+s\varphi(u),
\qquad u=\frac{c-\mu}{s}.
\]
Differentiating with respect to $s$ gives
\[
\frac{\partial}{\partial s}\EI(c;\mu,s)=\varphi(u)\ge0.
\]
Thus each achievement-weight contribution is nondecreasing in its predictive standard deviation. Finite sums, quadrature sums, and integrals with nonnegative weight density preserve this monotonicity.
\end{proof}

\begin{proof}[Proof of independent objective-Gaussian integrand and conditional variance statement]
Use the layer-cake identity
\[
\E[(c-X)_+]=\int_{-\infty}^c\Prob(X\le t)\dd t
\]
with $X=g_\lambda(Y(x);z^+)$. The event $g_\lambda(Y(x);z^+)\le t$ is equivalent to $Y_i(x)\le z_i^+ + t/\lambda_i$ for all interior-weight coordinates. Independence factorizes the probability into the product of one-dimensional Gaussian distribution functions, proving the formula.

For variance monotonicity, changing one objective standard deviation differentiates one Gaussian distribution factor in the product. The sign of this derivative depends on the standardized threshold $z_i^+ + t/\lambda_i-\mu_i(x)$. If the integration threshold $h_A(\lambda)$ lies below the scaled mean contribution for the varied coordinate, the derivative has the favourable sign over the integration range and monotonicity holds. If this condition is violated, positive and negative derivative regions may coexist, so the integral need not be monotone. This proves the conditional nature of the objective-Gaussian variance result.
\end{proof}

\begin{proof}[Proof of the discrete achievement-space ER2I algorithm proposition]
For each weight $\lambda^{(k)}$, the current best achievement level is
\[
h_{n,k}=\min_i g_{\lambda^{(k)}}(f(x_i);z^+).
\]
A candidate $x$ improves this scalar level by
\[
(h_{n,k}-Z_k(x))_+.
\]
Under the scalar Gaussian predictive model $Z_k(x)\sim N(m_k(x),s_k^2(x))$, the expected improvement is the closed-form scalar EI. Since the discrete $R_2$ acquisition is the average over the $K$ weights, the acquisition is the average of these $K$ scalar EI values. Updating $h_{n,k}$ after a new observation only requires comparing the old value with the new achievement value at that weight. This gives the stated algorithmic implementation.
\end{proof}

\begin{proof}[Proof of numerical quadrature theorem]
The exact achievement-space ER2I is an integral of the scalar EI integrand over $\DeltaM$. Replacing this integral by any quadrature rule $Q_L=\{(\lambda^{(\ell)},w_\ell)\}$ gives the stated approximation. The error is the quadrature error for the function $\lambda\mapsto\EI(h_A(\lambda);m_x(\lambda),s_x(\lambda))\rho(\lambda)$.
\end{proof}

\end{document}